\documentclass[12pt]{amsart}
\usepackage{amsmath,times,epsfig,amssymb,amsbsy,amscd,amsfonts,amstext,color,bm}
\usepackage[margin=1in]{geometry}
\usepackage{enumerate}
\usepackage{placeins}
\usepackage{array}
\usepackage{tabulary}
\usepackage{multirow}
\usepackage{graphicx}
\usepackage{hyperref}
\usepackage{hhline}
\usepackage{soul}
\usepackage{multirow}
\usepackage[table, svgnames, dvipsnames]{xcolor}
\usepackage{makecell, cellspace, caption}
\usepackage{subcaption}
\usepackage[color,matrix,arrow]{xypic}
\usepackage{tikz}
        \usetikzlibrary{arrows.meta}
        \usetikzlibrary{arrows}
        \usetikzlibrary{quotes}
        \usetikzlibrary{graphs}
        \usetikzlibrary{positioning}
\usepackage{tikz-cd}
\hypersetup{
  colorlinks   = true, 
  urlcolor     = blue, 
  linkcolor    = blue, 
  citecolor   = red 
}

\theoremstyle{plain}
\newtheorem{theorem}{Theorem}[section]
\newtheorem{corollary}[theorem]{Corollary}
\newtheorem{lemma}[theorem]{Lemma}
\newtheorem{claim}[theorem]{Claim}

\newtheorem{proposition}[theorem]{Proposition}

\makeatletter
    
    \@addtoreset{equation}{section}
  \makeatother

\theoremstyle{definition}
\newtheorem{definition}[theorem]{Definition}
\newtheorem{example}[theorem]{Example}
\newtheorem{conjecture}[theorem]{Conjecture}

\theoremstyle{remark}
\newtheorem{remark}[theorem]{Remark}

\newcommand{\A}{\mathcal{A}}
\newcommand{\scB}{\mathcal{B}}
\newcommand{\C}{\mathbb{C}}

\newcommand{\scC}{\mathcal{C}}

\newcommand{\Q}{\mathbb{Q}}
\newcommand{\R}{\mathbb{R}}
\newcommand{\scQ}{\mathcal{Q}}
\newcommand{\scR}{\mathcal{R}}

\newcommand{\scN}{\mathcal{N}}

\newcommand{\TT}{\mathbb{T}}
\newcommand{\Z}{\mathbb{Z}}

\newcommand{\h}{{\rm h}}

\newcommand{\asc}{{\rm asc}}
\newcommand{\dsc}{{\rm dsc}}

\newcommand{\D}{{\mathcal{D}}}

\newcommand{\GL}{\mathrm{GL}}
\newcommand{\M}{\mathcal{M}}

\newcommand{\lcm}{\operatorname{lcm}}

\newcommand{\quasi}{\operatorname{quasi}}

\newcommand{\rk}{\operatorname{rank}}
\newcommand{\Mat}{\operatorname{Mat}}

\newcommand{\diag}{\operatorname{diag}} 
\newcommand{\tor}{\operatorname{tor}} 
\newcommand{\ord}{\operatorname{ord}}

\newcolumntype{K}[1]{>{\centering\arraybackslash}p{#1}}

\begin{document}

\title[Period collapse in characteristic quasi-polynomials of hyperplane arrangements]{Period collapse in characteristic quasi-polynomials of hyperplane arrangements}

\begin{abstract}
\label{sec:intro}
Given an integral hyperplane arrangement, Kamiya-Takemura-Terao (2008 \& 2011) introduced the notion of characteristic quasi-polynomial, which enumerates the cardinality of the complement of the arrangement modulo a positive integer. The most popular candidate for period of the characteristic quasi-polynomials is the lcm period. In this paper, we initiate a study of period collapse in characteristic quasi-polynomials stemming from the concept of period collapse in the theory of Ehrhart quasi-polynomials. We say that period collapse occurs in a characteristic quasi-polynomial when the minimum period is strictly less than the lcm period. Our first main result is that in the non-central case, with regard to period collapse anything is possible: period collapse occurs in any dimension $\ge 1$, occurs for any value of the lcm period $\ge 2$, and the minimum period when it is not the lcm period can be any proper divisor of the lcm period. Our second main result states that in the central case, however, no period collapse is possible in any dimension, that is, the lcm period is always the minimum period. 
 \end{abstract}

 \author{Akihiro Higashitani}
\address{Department of Pure and Applied Mathematics, Graduate School of Information Science and Technology, Osaka University, Suita, Osaka 565-0871, Japan}
\email{higashitani@ist.osaka-u.ac.jp}
\author{Tan Nhat Tran}
\address{Tan Nhat Tran, Department of Mathematics, Hokkaido University, Kita 10, Nishi 8, Kita-Ku, Sapporo 060-0810, Japan \newline
Current address: Fakult\"at f\"ur Mathematik, Ruhr-Universit\"at Bochum, D-44780 Bochum, Germany}
\email{tan.tran@ruhr-uni-bochum.de}
\author{Masahiko Yoshinaga}
\address{Masahiko Yoshinaga, Department of Mathematics, Hokkaido University, Kita 10, Nishi 8, Kita-Ku, Sapporo 060-0810, Japan.}
\email{yoshinaga@math.sci.hokudai.ac.jp}


\date{\today}
\maketitle

\tableofcontents

\section{Introduction}
Let $C$ be an $m \times n$ matrix with integral entries and let $b \in\Z^n$ be an integral vector. 
The matrix $A= \begin{pmatrix} C\\ b\end{pmatrix}$ defines an arrangement $\A$ of $n$ hyperplanes $H_j$ ($1 \le j \le n$) in $\R^m$.
Let $q\in\Z_{>0}$ and denote $\Z_q:=\Z/q\Z$.  
 For $a \in \Z$, let $[a]_q := a + q\Z \in \Z_q$ denote the $q$ reduction of $a$. 
For a matrix or vector $A'$ with integral entries, denote by $[A']_q$ the entry-wise $q$ reduction of $A'$.
The matrix $[A]_q$ defines an arrangement $\A_q$ of $n$ subsets\footnote{These subsets become subgroups in the central case.}  $H_{j,q}$ ($1 \le j \le n$) in $\Z_q^m$.
We call the case $b = 0$ the \emph{central case}, and call the case $b \ne 0$ the \emph{non-central case}.\footnote{Kamiya-Takemura-Terao \cite{KTT11} used the term ``the non-central case" when $b$ is any integral vector.} 
Kamiya-Takemura-Terao \cite{KTT08, KTT11} showed that there exist integers $q_0\in\Z_{\ge0}$ and $\rho\in\Z_{>0}$ such that the cardinality of the complement $\M(\A_q)$ of $\A_q$ is a quasi-polynomial in $q$ with period $\rho$ for all $q>q_0$.
They called this quasi-polynomial the \emph{characteristic quasi-polynomial}, and called the period $\rho$ the \emph{lcm period}.

 Given a rational polytope $P$ in $\R^m$, the \emph{denominator} $D(P)$ of $P$ is the smallest $k \in\Z_{>0}$ for which the vertices of $kP$ are in $\Z^m$. 
Ehrhart \cite{E62} proved that  ${\rm L}_{P}(q):=\#(qP \cap \Z^m)$ is a quasi-polynomial  in $q$ with period $D(P)$. 
We call ${\rm L}_{P}(q)$ the \emph{Ehrhart quasi-polynomial} of $P$. 
The denominator $D(P)$ is  in general not the minimum period  of ${\rm L}_{P}(q)$, and when $D(P)$ is strictly greater than the minimum period, we say that \emph{period collapse} occurs. 
This phenomenon was first systematically studied by McAllister-Woods \cite{MW05}, after an example of a $3$-dimensional pyramid $P$ with denominator $D(P) = 2$ but with the minimum period $1$ has been found by Stanley \cite{St97}.
 McAllister-Woods proved that period collapse never occurs in dimension $1$, however in dimension $2$ and higher, anything is possible: period collapse occurs in any dimension $\ge 2$, occurs for any denominator $\ge 2$, and the minimum period when it is not the denominator can be any proper divisor of the denominator \cite[Theorems 2.1 and 2.2]{MW05}. 
 
Kamiya-Takemura-Terao provided two different proofs for showing that $\#\M(\A_q)$ is indeed a quasi-polynomial. The first proof uses the theory of elementary divisors. The second proof relies on an expression of $\#\M(\A_q)$ as a sum of the Ehrhart quasi-polynomials of polytopes cut out of the unit cube  by the deformations of the hyperplanes in $\A$ \cite[\S2]{KTT08}, \cite[\S3]{KTT11}, or as the Ehrhart quasi-polynomial of an  ``inside-out'' polytope in the sense of Beck-Zaslavsky \cite{BZ06}. 
Both methods are powerful and interesting. 
The former gives explicitly the coefficients of the quasi-polynomials which in turn produces the lcm period the best known candidate for period so far, and it also opens an unexpected connection with the theory of toric arrangements  \cite{TY19, LTY21}. 
The latter, generally, applies only to the central case (see \S\ref{sec:Ehrhart}) but enables one to apply many results in the Ehrhart theory to further discover interesting properties of the arrangement. 
The most notable example that the Ehrhart theory approach is applicable to the non-central case is when the vector configuration arises from root systems. 
For example, it is used to affirmatively settle the ``Riemann hypothesis" a conjecture of Postnikov-Stanley on the roots of the characteristic polynomials of the  extended Linial arrangements \cite{Y18W, Y18L, Tam20}, or used to characterize a geometric-enumerative property of the arrangement the so-called Worpitzky-compatibility \cite{ATY20, TT21}.

Inspired by the notion of period collapse in the Ehrhart theory, in this paper we say that \emph{period collapse} occurs in a characteristic quasi-polynomial when the minimum period is strictly less than the lcm period.
Some examples of such period collapse in the non-central case were found recently. 
The third author  applied the Worpitzky partition of root system's fundamental parallelepiped to compute the  characteristic quasi-polynomials of several deformed Weyl arrangements. 
In particular, he showed that the characteristic quasi-polynomial of any extended Shi arrangement has minimum period $1$ \cite{Y18W}, where previous computation explicitly shows lcm period $> 1$ in all cases but type $A$, e.g., \cite{Su98, KTT10}. 
Tamura \cite{Tam20} enhanced the above calculation in the case of extended Linial arrangements and showed that period collapse occurs in the corresponding characteristic quasi-polynomial for certain parameters  (see \S\ref{sec:rs} for more details).

Thus we know that period collapse with respect to the characteristic quasi-polynomials of non-central arrangements occurs  in any dimension $\ge1$, but can period collapse occur for any value of the lcm period? Can the minimum period be any proper divisor of the lcm period when it is not the lcm period? 
Our first main result in this paper given below will affirmatively answer these questions. 

\begin{theorem} 
\label{thm:main} 
 Given $m \ge 1$, and given $p\ge1$ and $s\ge1$ such that $s\mid p$, there exists a non-central hyperplane arrangement in $\R^m$ whose characteristic quasi-polynomial has lcm period $p$ but has minimum period $s$.
\end{theorem}

\noindent
We also provide two proofs of  Theorem \ref{thm:main}. 
The first proof using the elementary divisors method of Kamiya-Takemura-Terao is standard and direct. 
The second proof based on the concept of $\GL_m(\Z)$-equidecomposable polytopes encompasses more important aspects of our discussion. 
The approach we used to describe the Ehrhart-theoretic formula of $\#\M(\A_q)$ is different and ``complement" to the approach above: instead of counting the lattice points in the polytopes, we count the lattice points in the deformations of the hyperplanes in $\A$ within the unit cube. 
In addition, it gives a new class of non-central arrangements not arising from root systems where the Ehrhart theory can be effectively applied. 

 Determining the minimum period of a quasi-polynomial is natural but in general a challenging problem. 
Since their introduction, the methods above motivated several discussions on the minimum period of the characteristic quasi-polynomial in the central case, e.g., \cite{KTT11, CW12, Tan20}, as well as help to derive many computational results, e.g., \cite{KTT10, T19} that support the equality between the  minimum period and the  lcm period. 
 However, to date, no proof has been known. 
 Recently, a complete description of the constituents of the characteristic quasi-polynomial was found via the theory of toric arrangements due to the second author and the third author \cite{TY19}. 
Using this approach, we are able to affirmatively prove that this is always the case.
 \begin{theorem} 
\label{thm:central} 
If $C$ is an $m \times n$ integral matrix with nonzero columns for $m, n\ge 1$, then the lcm period of the characteristic quasi-polynomial of the matrix $\begin{pmatrix} C\\ 0\end{pmatrix}$  is always the minimum period.
\end{theorem}

Studying periodicity behavior and period collapse is an active area in the Ehrhart theory, for some recent discussions and developments see, e.g., \cite{BSW08, MM17, KW18}. 
One may investigate many similar aspects regarding period collapse in the characteristic quasi-polynomials.

The remainder of the paper is organized as follows. 
In \S \ref{sec:CQP}, following \cite[\S3]{KTT11}, we recall the definition of integral hyperplane arrangement modulo a positive integer and basic facts on the corresponding characteristic quasi-polynomial via elementary divisors. 
In \S \ref{sec:main}, we give the first proof of Theorem \ref{thm:main}  by using the formula of the characteristic quasi-polynomial from \S\ref{sec:CQP}. 
In \S \ref{sec:Ehrhart}, we give the second proof of Theorem \ref{thm:main} via Ehrhart theory. 
In \S \ref{sec:rs}, we discuss period collapse in characteristic quasi-polynomials of certain subarrangements of the extended Shi and Linial arrangements.
In \S \ref{sec:open}, we give a proof of Theorem \ref{thm:central}.
 \section{Recollection of characteristic quasi-polynomials}
\label{sec:CQP}

A function $\varphi: \Z \to \C$ is called a \emph{quasi-polynomial} if there exist a positive integer $\rho\in\Z_{>0}$ and polynomials $f^k(t)\in\Q[t]$ ($1 \le k \le \rho$) such that for any $q\in\Z_{>0}$ with  $q\equiv k\bmod \rho$, 
\begin{equation*}
\varphi(q) =f^k(q).
\end{equation*}
The number $\rho$ is called a \emph{period} and the polynomial $f^k(t)$ is called the \emph{$k$-constituent} of the quasi-polynomial $\varphi$. 
Equivalently, $\varphi$ is of the form
$$\varphi(q) = d_m(q)q^m + d_{m-1}(q)q^{m-1} + \dots +d_0(q),$$
where each $d_i:\Z \to \Q$ is a periodic function with integral period, i.e., $d_i(q)=d_i(q+\tau_i)$ for $\tau_i \in \Z$. 
In this case, the above period $\rho$ can be chosen as $\rho = \lcm(\tau_0,\tau_1,\ldots,\tau_m)$.
The smallest such $\rho$ is called the \emph{minimum period} of the quasi-polynomial $\varphi$. 
The minimum period is necessarily a divisor of any period. 
It is not hard to see that if $\tau_i$ is the minimum period of $d_i$ for each $i$, then $\lcm(\tau_0, \ldots,\tau_m)$ is the minimum period of $\varphi$.

A quasi-polynomial $\varphi$ with a period $\rho$  is said to have the \emph{gcd property} with respect to $\rho$ if the $k$-constituent $f^k(t)$ ($1 \le k \le \rho$) depends on $k$ only through $\gcd(k,\rho)$, i.e., $f^a(t)=f^{b}(t)$ if $\gcd(a, \rho)=\gcd(b, \rho)$. 

Let $m,n\in\Z_{>0}$ be positive integers. 
Denote by $\Mat_{m\times n}(\Z)$ the set of all $m \times n$ matrices with integer entries. 
Let $C=(c_1,\ldots,c_n) \in \Mat_{m\times n}(\Z)$ with $c_j =(c_{1j},\ldots,c_{mj})^T \ne (0,\ldots,0)^T$ ($1\le j\le n$) and let $b=(b_1,\ldots,b_n) \in\Z^n$. 
Here, $^T$ stands for the transpose of a matrix. 
Set $A = \begin{pmatrix} C\\ b\end{pmatrix}  \in \Mat_{(m + 1)\times n}(\Z)$. 
The matrix $A$ defines the following hyperplane arrangement in $\R^m$
$$\A =\A_A := \{ H_j : 1\le j\le n\},$$
where
$$H_j = H_{c_j}  := \{x=(x_1,\ldots,x_m)\in \R^m: xc_j=b_j\}.$$
Let $q\in\Z_{>0}$. The \emph{$q$ reduction} $\A_q$ of $\A$ is defined by 
$$\A_q := \{ H_{j,q}: 1\le j\le n\},$$
where
$$H_{j,q} := \{ z=(z_1,\ldots,z_m)\in \Z_q^m: z[c_j]_q=[b_j]_q\}.$$
Denote $\Z_q^{\times}:=\Z_q \setminus \{0\}$.
The complement $\M(\A_q)$ of $\A_q$ is defined by 
$$\M(\A_q) := \Z_q^m \setminus\bigcup_{j=1}^n H_{j,q} = \{z\in \Z_q^m: z[c_j]_q-[b_j]_q \in (\Z_q^{\times})^n\}.$$

For integers $a\le b$ and $n\ge1$, denote $[a,b]:=\{k\in\Z\mid a \le k\le b\}$ and $[n]:=[1,n]$. 
For $\emptyset \ne J \subseteq [n]$, define $A_J := \begin{pmatrix} C_J\\ b_J\end{pmatrix}  \in \Mat_{(m + 1)\times \#J}(\Z)$, where $C_J \in \Mat_{m\times \#J}(\Z)$ is the submatrix of $C$ consisting of the columns indexed by $J$ and $b_J \in  \Z^{\#J}$ is the subvector of $b$ consisting of the entries indexed by $J$.
Set $\ell(J):=\rk C_J$ and $\ell'(J):=\rk A_J$.
Let $e_{J,1} \mid e_{J,2} \mid \cdots \mid e_{J,\ell(J)}$ be the elementary divisors of $C_J$ and let $e'_{J,1} \mid e'_{J,2} \mid \cdots \mid e'_{J,\ell'(J)}$  be the elementary divisors of $A_J$.
Define the \emph{lcm period} of $C$ by
$$\rho_C:= \lcm (e_{J,\ell(J)}: \emptyset \ne J \subseteq [n]).$$
Define
$$q_0:= \max \{e_{J,\ell'(J)}: \ell'(J)=\ell(J)+1,  \emptyset \ne J \subseteq [n] \}.$$
When $\{ \emptyset \ne J \subseteq [n]:  \ell'(J)=\ell(J)+1 \} = \emptyset $, we agree that $q_0=0$.
Define
$$
\tilde{d}_J(q):=
\begin{cases}
\prod_{j=1}^{\ell(J)} \gcd(e_{J, j}, q) & \mbox{if $\gcd(e_{J, j}, q)=\gcd(e'_{J, j},q)$ for all $1 \le j \le \ell(J)$}, \\
0& \mbox{otherwise}.
\end{cases}
$$

\begin{theorem}{\cite[Theorem 3.1]{KTT11}}
\label{thm:KTT}
There exists a monic quasi-polynomial $\chi^{\quasi}_{\A}(q)$  with period $\rho_{C}$ having the gcd property w.r.t. $\rho_{C}$ such that $\#\M(\A_q)  = \chi^{\quasi}_{\A}(q)$ for all $q>q_0$. 
This quasi-polynomial is called the \emph{characteristic quasi-polynomial} of $\A$ (or of $A$). 
More precisely, the following holds: 
\begin{equation*}
\label{eq:KTT-thm} 
\chi^{\quasi}_{\A}(q) =q^m+ \sum_{J:\,\,\emptyset \ne J \subseteq [n],\,  \ell'(J)=\ell(J)} (-1)^{\#J} \tilde{d}_J(q)q^{m-\ell(J)}.
\end{equation*}
\end{theorem}

The name ``characteristic quasi-polynomial" is made by inspiration of the following fact. 
Let $\chi_\A( t)$ denote the \emph{characteristic polynomial} (e.g., \cite[Definition 2.52]{OT92}) of $\A$.
\begin{theorem}[{e.g., \cite[Remark 3.3]{KTT11}}]
\label{thm:KTT11} 
Let  $f_\A^k(t)\in\Z[t]$ ($1 \le k \le \rho_C$) denote the  $k$-constituent of $\chi^{\quasi}_{\A}(q)$. 
The following holds: 	
$$f^1_{\A}(t)=\chi_{\A}(t).$$
\end{theorem}

 \section{Proof of the first main result: the non-central case}
\label{sec:main} 
In this section, we prove the first main result (Theorem~\ref{thm:main}) of the paper by using the formula of the characteristic quasi-polynomial derived from the elementary divisors method (Theorem \ref{thm:KTT}). 
We will use in the proof several notations introduced in the previous section.
\begin{proof}[\textbf{Proof of  Theorem \ref{thm:main}}]
Given $m\ge 1$, and given $p\ge1$ and $s\ge1$ such that $s\mid p$, let $A=(a_1, \ldots,a_{p+m})$ be an $(m + 1) \times (p+m)$ integral matrix  defined by
$$
A = 
\begin{pmatrix}
    1 & \cdots & 0& 0 & 1 & 1  & \cdots  & 1 \\
    \vdots & \ddots & \vdots & \vdots  & \vdots&\vdots &  & \vdots \\
   0 & \cdots & 1 & 0 & 1 & 1  &\cdots  & 1 \\
      0 & \cdots & 0 & s & p & p  & \cdots  & p \\
       0 & \cdots &  0 &  0 &  1 &  2  & \cdots  &   p 
       \end{pmatrix}.
$$
Thus matrix $A$ defines a non-central hyperplane arrangement $\A$ in $\R^m$. 
We show that the characteristic quasi-polynomial $\chi^{\quasi}_{\A}(q)$ of $\A$ has lcm period $p$ but has minimum period $s$. 
More explicitly, we show the following formula
\begin{equation}
\label{eq:explicit} 
\chi^{\quasi}_{\A}(q) = \sum_{k=0}^{m} (-1)^k \left[\left(p\binom{m-1}{k-2}+\binom{m-1}{k-1}\right)\gcd(q,s)+ p\binom{m-1}{k-1}+\binom{m-1}{k}  \right]q^{m-k}.
\end{equation}

Let $C=(c_1, \ldots,c_{p+m})$ be the $m \times (p+m)$ integral matrix obtained by removing the last row of matrix $A$. 
By Theorem \ref{thm:KTT},
\begin{equation}
\label{eq:KTT} 
\chi^{\quasi}_{\A}(q) =q^m+ \sum_{J:\,\,\emptyset \ne J \subseteq [p+m],\, \ell'(J)=\ell(J)} (-1)^{\#J} \tilde{d}_J(q)q^{m-\ell(J)}.
\end{equation}
We classify the subsets $J$ such that $\emptyset \ne J \subseteq [p+m]$ into two cases: (a) $\emptyset \ne J \subseteq [m]$, and (b) $J=J_1 \sqcup J_2$ where $J_1\subseteq [m], \emptyset \ne J_2\subseteq [m+1,p+m]$. 

Case (a): $\emptyset \ne J \subseteq [m]$. Every subset $J$ in this case satisfies $\ell'(J)=\ell(J)$ and $1 \le \ell(J) \le m$. 
Clearly, $A_J$ and $C_J$ have the same elementary divisors.
If $\ell(J)=k$ then $\#J=k$ since any $k$ columns of the matrix $(c_1, \ldots,c_{m})$ are linearly independent. 
If $a_m=(0,\ldots,0,s,0)^T$ is not a column of $A_J$, then obviously, $A_J$ has the elementary divisors $1,\ldots,1$, hence $\tilde{d}_J(q)=1$.
Otherwise, $A_J$ has the elementary divisors $1,\ldots,1,s$ in which case $\tilde{d}_J(q)=\gcd(q,s)$. 
Thus these subsets contribute to $\chi^{\quasi}_{\A}(q)$ the following quantity
\begin{equation}
\label{eq:case-a} 
\sum_{k=1}^{m} (-1)^k \left[\binom{m-1}{k}+\binom{m-1}{k-1}  \gcd(q,s) \right]q^{m-k}.
\end{equation}
 
Case (b): $J=J_1 \sqcup J_2$ where $J_1\subseteq [m], \emptyset \ne J_2\subseteq [m+1,p+m]$. 
Since $\ell'(J)=\ell(J)$, the matrix $A_J$ can not have two distinct columns $a_i, a_j$ for $i, j \in [m+1,p+m]$. 
By the same reason, it can not happen that $J_1= [m]$.
Therefore, $J_1\subsetneq [m]$ and $|J_2|=1$. 
Thus, $A_J$ has linearly independent columns, so does $C_J$. 
Note that $1 \le \ell(J) \le m$. 
If $a_m$ is a column of $A_J$, then $A_J=(A_{J_1} ,A_{J_2})$ and $C_J=(C_{J_1} ,C_{J_2})$ have the same elementary divisors $1,\ldots,1,s$ (regardless of the value of  $\ell(J)$). 
The crucial point here is that the column $A_{J_2}=(1,\ldots,1,p,r)^T$ (for some $1 \le r \le p$) has all 1's at its first $m-1$ entries, while the matrix $A_{J_1}$ consists of at most $m-2$ columns of the matrix $(c_1, \ldots,c_{m-1})$. 
It follows that $\tilde{d}_J(q)=\gcd(q,s)$. 
If $1 \le \ell(J) \le m-1$ and $a_m$ is not a column of $A_J$, then by a similar reason, $A_J$ and $C_J$ have the same elementary divisors $1,\ldots,1$. 
Thus the subsets $J$ with $1 \le \ell(J) \le m-1$ contribute to $\chi^{\quasi}_{\A}(q)$ the following quantity
\begin{equation}
\label{eq:case-b1} 
\sum_{k=1}^{m-1} (-1)^kp \left[\binom{m-1}{k-1}+\binom{m-1}{k-2}  \gcd(q,s) \right]q^{m-k}.  
\end{equation}
If  $\ell(J)=m$ and $a_m$ is not a column of $A_J$, then obviously,  $A_J=(c_1, \ldots,c_{m-1},A_{J_2})$. 
Thus $C_J$ has the elementary divisors $1,\ldots,1,p$, while $A_J$ has the elementary divisors $1,\ldots,1,\gcd(p,r)$. 
It is a trivial fact that if $a \mid p$, then $\#\{r \in [p]:a\mid r\}= \frac{p}{a}$.
Since the $\gcd$ is an associative function: $\gcd(q, \gcd(p, r)) = \gcd(\gcd(q, p), r)$, we have  $\#\{r \in [p]: \gcd(q, p) = \gcd(q, \gcd(p, r))\}=\#\{r \in [p]: \gcd(q, p) \mid r\}= \frac{p}{\gcd(q, p)}$.
Thus the subsets $J$ with $\ell(J)=m$ contribute to $\chi^{\quasi}_{\A}(q)$ the following quantity
\begin{equation}
\label{eq:case-b2} 
 (-1)^m \left[ p(m-1)\gcd(q,s) + p\right].  
 \end{equation}

Combining Eq. \eqref{eq:KTT}-\eqref{eq:case-b2}, we obtain Eq. \eqref{eq:explicit}. 
Along the proof, we see that $\chi^{\quasi}_{\A}(q)$ has lcm period $p$. 
The fact that $s$ is a period of $\chi^{\quasi}_{\A}(q)$ is clear from Eq. \eqref{eq:explicit} because each coefficient is a periodic function with period $s$: $\gcd(q,s)=\gcd(q+s,s)$. 
Finally, we show that $s$ is indeed the minimum period.
Let $s'$ be the minimum period. Note that $s' \mid s$. 
Let $d_0(q)$ denote the constant coefficient of $\chi^{\quasi}_{\A}(q)$. 
Thus $d_0(q) =  (-1)^m \left[(p(m-1)+1)\gcd(q,s)+p\right]$.
Since $s'$ is a period, $d_0(q)=d_0(q+s')$ which implies that $\gcd(q,s)=\gcd(q+s',s)$ for all $q>q_0$.
Let $q=Ms$ for sufficiently large $M$. 
This implies that $s=\gcd(Ms,s)=\gcd(Ms+s',s)=\gcd(s',s)=s'$.

\end{proof}

 \begin{remark}
\label{rem:slightly}
 Given $m\ge 2$, and given $p\ge2$ and $s\ge1$ such that $s\mid p$, let $B$ be an $(m + 1) \times (p+m)$ integral matrix  defined by
$$
B = 
\begin{pmatrix}
    1 & \cdots & 0& 0 & 1 & 1 & 1 & \cdots  & 1 \\
    \vdots & \ddots & \vdots & \vdots & \vdots & \vdots&\vdots &  & \vdots \\
   0 & \cdots & 1 & 0 & 1 & 1 & 1 &\cdots  & 1 \\
      0 & \cdots & 0 & s & p & p & p & \cdots  & p \\
      0 & \cdots & 0 & 0 & 0 & 1 & 2 & \cdots  & p-1 
\end{pmatrix}.
$$
By the same method as in the proof of Theorem \ref{thm:main}, one can show that the non-central hyperplane arrangement $\scB$ in $\R^m$ defined by matrix $B$ has lcm period $p$ but has minimum period $s$. 
More explicitly, the following formula holds: 
\begin{equation*}
\label{eq:explicit-B} 
\begin{aligned}
\chi^{\quasi}_{\scB}(q) & = \sum_{k=0}^{m-1} (-1)^k \left[\left(p\binom{m-1}{k-2}+\binom{m-1}{k-1}\right)\gcd(q,s)+ p\binom{m-1}{k-1}+\binom{m-1}{k}  \right]q^{m-k} \\ 
& +p(-1)^m[1+(m-1)\gcd(q,s)],
\end{aligned}
\end{equation*}
which differs from $\chi^{\quasi}_{\A}(q)$ only at the constant coefficient: $\chi^{\quasi}_{\A}(q) - \chi^{\quasi}_{\scB}(q) = (-1)^m\gcd(q,s)$.
However, matrix $B$ has a slightly small drawback. 
We understand that the matrix $B$ when $m=1$ is given by
$$
B = 
\begin{pmatrix}
s & p & p &   \cdots  & p \\
0 & 0 & 1 &   \cdots  & p-1 
\end{pmatrix}.
$$
The corresponding characteristic quasi-polynomial has lcm period $p$  but has minimum period $1$:  $\chi^{\quasi}_{\scB}(q) = q-p$.
Also if $p=1$, regardless of the value of $m$, the last row of matrix $B$ is identical to the zero vector, hence $\scB$ is a central hyperplane arrangement. 

\end{remark}

\begin{remark}
\label{rem:more-general}
 Given $m\ge 1$, $a\ge 1$, $s\ge 1$, and  $p\ge1$, let $A'$ be an $(m + 1) \times (p+m)$ integral matrix  defined by
$$
A' = 
\begin{pmatrix}
    1 & \cdots & 0& 0 & 1 & 1  & \cdots  & 1 \\
    \vdots & \ddots & \vdots & \vdots  & \vdots&\vdots &  & \vdots \\
   0 & \cdots & 1 & 0 & 1 & 1  &\cdots  & 1 \\
      0 & \cdots & 0 & s & a & a  & \cdots  & a \\
      0 & \cdots & 0 & 0 & 1 & 2  & \cdots  & p 
\end{pmatrix}.
$$
By the same method as in the proof of Theorem \ref{thm:main}, one can show that $A'$ determines a non-central hyperplane arrangement $\A'$ in $\R^m$ whose characteristic quasi-polynomial $\chi^{\quasi}_{\A'}(q)$ has lcm period $\lcm(s,a)$ and  
\begin{equation*}
\label{eq:A-A'} 
\chi^{\quasi}_{\A}(q) - \chi^{\quasi}_{\A'}(q) = (-1)^m\left(p - \gcd(q, a) \cdot \#\{r \in [p]: \gcd(q, a) \mid r\}  \right).
\end{equation*}
In particular, if $a=p$ or $ \gcd(q, a)=1$ (e.g., $a=1$), then $\chi^{\quasi}_{\A}(q) = \chi^{\quasi}_{\A'}(q)$.
\end{remark}

 \section{Period collapse and Ehrhart theory}
\label{sec:Ehrhart} 

In this section, we give the second proof of Theorem \ref{thm:main} using techniques from the Ehrhart theory. 
First let us recall how the characteristic quasi-polynomial $\chi^{\quasi}_{\A}(q)$ (notation from \S \ref{sec:CQP}) and the Ehrhart theory of polytopes are related following, e.g., \cite[\S2.2]{KTT08}.

We use the notation $U_\R$ to denote an interval $U$ of $\R$, e.g., $[0,1]_\R$ denotes the unit interval in $\R$ and $[0,1]_\R^d$ denotes the $d$-dimensional unit cube in $\R^d$. 
Apply the bijection $(0,q]_\R\cap\Z \simeq \Z_q$, we can write
\begin{equation}
\label{eq:Ehrhart-def}
\begin{aligned}
\chi^{\quasi}_{\A}(q) 
& = \#\{z\in \Z_q^m: z[c_j]_q-[b_j]_q \in (\Z_q^{\times})^n\}\\
& = \# \left(\Z^m\cap (0,q]_\R^m  \setminus \bigcup_{1\le j \le n,\, k\in \Z}\{x \in  \R^m : xc_j=kq+b_j \} \right).
\end{aligned}
\end{equation}
In the central case $b=0$, we can further write
$$\chi^{\quasi}_{\A}(q)  = \# \left(\Z^m\cap \left(q \times \left((0,1]_\R^m  \setminus \bigcup_{1\le j \le n,\, k\in \Z}\{x \in  \R^m  : xc_j=k \} \right)\right)\right).$$
We see that the half-open unit cube $(0,1]_\R^m$ cut out by the hyperplanes $xc_j=kq+b_j$, $1\le j \le n,\, k\in \Z$ is a disjoint union of rational ``partially open" polytopes. 
Thus the Ehrhart theory can be applied to some faces of each polytope to show that $\chi^{\quasi}_{\A}(q)$ is a sum of the Ehrhart quasi-polynomials. 
In the non-central case $b\ne0$, however, it is not straightforward to derive the similar conclusion. 
The nicest setting we are able to do so  is when the arrangement $\A$ comes from root systems. In this case, the disjoint union above induces a partition of the cube into simplices the so-called Worpitzky partition, e.g., \cite{A04, Y18W}, and the Ehrhart quasi-polynomial of each partially open simplex can be described by using the (generalized) Eulerian statistics, e.g., \cite{LP18,  ATY20} (see \S\ref{sec:rs} for more details). 

Our approach in this paper is different and ``complement" to the approach above: instead of counting the lattice points in the simplices, we describe the lattice points in the intersections of $(0,1]_\R^m$  and the hyperplanes $xc_j=kq+b_j$. It turns out that in our setting these lattice points can be nicely placed in the unit cubes in lower dimensions after applying affine unimodular transformations.

If $P$ is a rational polytope or $P$ is the relative interior of a rational polytope in $\R^m$, we continue to use the notation ${\rm L}_{P}(q)=\#(qP \cap \Z^m)$ for the Ehrhart quasi-polynomial of $P$. 
We are now in position to state the main result in this section for which we will give an Ehrhart-theoretic proof. 
\begin{theorem} 
\label{thm:main-Ehrhart} 
Let $A$ be the matrix defined in the proof of Theorem \ref{thm:main} (\S\ref{sec:main}). The following holds: 
\begin{equation}
\label{eq:Ehrhart}
\chi^{\quasi}_\A(q)=(q-\gcd(q,s))\left({\rm L}_{(0,1)_\R^{m-1}}(q)+p\sum_{k=1}^{m-1}(-1)^k{\rm L}_{(0,1)_\R^{m-1-k}}(q)\right) + (-1)^m p. 
\end{equation}
\end{theorem}

Before proving Theorem \ref{thm:main-Ehrhart}, let us mention that  Eq. \eqref{eq:explicit} and Eq. \eqref{eq:Ehrhart} are essentially equivalent.
\begin{proposition} 
\label{prop:another} 
RHS \eqref{eq:explicit} $=$ RHS \eqref{eq:Ehrhart}.
\end{proposition}
\begin{proof}
First observe that for any $d \geq 1$, 
\begin{align}\label{eq:easy}
\sum_{k=0}^d(-1)^k\binom{d}{k-1}q^{d-k}=\sum_{k=1}^d(-1)^k(q-1)^{d-k}.
\end{align}
Denote $g:=\gcd(q,s)$.
RHS \eqref{eq:explicit} is given by
\begin{align*}
& \sum_{k=0}^{m} (-1)^k g\left(p\binom{m-1}{k-2}+\binom{m-1}{k-1}\right)q^{m-k} + \sum_{k=0}^m (-1)^k\left(p\binom{m-1}{k-1}+\binom{m-1}{k}\right)q^{m-k} \\
=&-g\sum_{k=1}^m (-1)^{k-1} \left(p\binom{m-1}{k-2}+\binom{m-1}{k-1}\right)q^{m-k} \\
&\quad\quad\quad + q\sum_{k=0}^{m-1} (-1)^k\left(p\binom{m-1}{k-1}+\binom{m-1}{k}\right)q^{m-1-k} + (-1)^mp \\
=&(q-g)\left(\sum_{k=0}^{m-1} (-1)^k \left(\binom{m-1}{k}+p\binom{m-1}{k-1}\right)q^{m-1-k}\right) + (-1)^mp.
\end{align*}
Note that  ${\rm L}_{(0,1)_\R^d}(q)=(q-1)^d$ for $d \ge 1$ and we agree that it equals $1$ when $d=0$. 
Apply Eq. \eqref{eq:easy}.
\end{proof}

Let $\GL_m(\Z)$ denote the set of all $m\times m$ \emph{unimodular} matrices (integral matrices having determinant $+1$ or $-1$).
Now let us recall the notion of $\GL_m(\Z)$-equidecomposability following, e.g., \cite[Definition 3.1]{HM08}.
Let ${\rm Aff}_m(\Z):= \GL_m(\Z)\ltimes\Z^m$ be the group of affine unimodular transformations on $\R^m$.
Two (open) polytopes $P, Q \subseteq\R^m$ are said to be \emph{unimodularly equivalent} if there exists $\sigma \in {\rm Aff}_n(\Z)$ such that $\sigma(P) = Q$.
Two sets $P, Q \subseteq\R^m$ are said to be \emph{$\GL_m(\Z)$-equidecomposable} if there are open polytopes $T_1,\ldots, T_n$ and affine unimodular transformations $\sigma_1,\ldots,\sigma_n \in {\rm Aff}_m(\Z)$ such that
$$ P = \bigsqcup_{i=1}^n T_i \quad \mbox{and}\quad Q= \bigsqcup_{i=1}^n \sigma_i(T_i).$$
Here, $ \bigsqcup$ indicates disjoint union. 
It is known that if $P, Q \subseteq\R^m$ are $\GL_m(\Z)$-equidecomposable, then $\#(qP \cap \Z^m)=\#(qQ \cap \Z^m)$ for all $q\in\Z_{>0}$.

Given $m\ge 1$, $a\ge 1$, and  $p\ge1$, let $D$ be an $(m + 1) \times (p+m)$ integral matrix  defined by
$$
D = 
\begin{pmatrix}
    1 & \cdots & 0& 0 & 1 & 1  & \cdots  & 1 \\
    \vdots & \ddots & \vdots & \vdots  & \vdots&\vdots &  & \vdots \\
   0 & \cdots & 1 & 0 & 1 & 1  &\cdots  & 1 \\
      0 & \cdots & 0 & 1 & a & a  & \cdots  & a \\
      0 & \cdots & 0 & 0 & 1 & 2  & \cdots  & p 
\end{pmatrix}.
$$
In other words, $D$ is a specialization of the matrix $A'$ in Remark \ref{rem:more-general} with $s=1$. 
Let $\D$ be the non-central  arrangement defined by matrix $D$.
In particular, $\chi^{\quasi}_\D(q)$ has lcm period $1$ hence $\chi^{\quasi}_\D(q)$ is a polynomial.
The  following formula is crucial for Theorem \ref{thm:main-Ehrhart} whose proof will be derived by the method of $\GL_m(\Z)$-equidecomposability.
\begin{lemma} 
\label{lem:s=1}
If $q>p$, then
\begin{equation}
\label{eq:s=1}
\begin{aligned}
\chi^{\quasi}_\D(q) & ={\rm L}_{(0,1)_\R^{m}}(q) + p\sum_{k=1}^{m-1}(-1)^k{\rm L}_{(0,1)_\R^{m-k}}(q) \\
&+ (-1)^m\#\left(\bigsqcup_{r=1}^p\bigsqcup_{\ell=1}^a\left\{x \in \R^2 : \begin{array}{c} x_1 =\ell q, 0< x_2< q, \\ x_1+ ax_2=a q+r\end{array} \right\}\cap \Z^2\right).
\end{aligned} 
\end{equation}
In particular, if either (i) $a=1$, or (ii) $q\ge 2p$ and $a=p$, then 
\begin{equation}
\label{eq:s=1-inp}
 \chi^{\quasi}_\D(q)={\rm L}_{(0,1)_\R^{m}}(q) + p\sum_{k=1}^{m}(-1)^k{\rm L}_{(0,1)_\R^{m-k}}(q).
 \end{equation}

\end{lemma}
\begin{proof}
By Eq. \eqref{eq:Ehrhart-def},  $\chi^{\quasi}_\D(q)$ is given by
\begin{equation}\label{eq:count}
 \chi^{\quasi}_\D(q) =
 \#\{x\in [q-1]^m : x_1+\cdots+x_{m-1}+ax_m \neq \ell q+r, \ell \in \Z, r\in[p]\}. 
\end{equation}
Note that we may assume $0 \leq \ell \leq a+m-2$ since $0< x_i< q$ for each $i$. 
Given $\ell,r \in \Z$, let $H_{\ell,r}^m:=\{x \in \R^m : x_1+\cdots+x_{m-1}+ax_m=\ell q+r\}$, and let $\scC_{\ell,r}^m:=H_{\ell,r}^m \cap (0,q)_\R^m$. 
Note that $\scC_{\ell,r}^m \cap \scC_{\ell',r}^m=\emptyset$ for any $r$ if $\ell\ne\ell'$, and $\scC_{\ell,r}^m \cap \scC_{\ell',r'}^m=\emptyset$ for any $\ell,\ell'$ if $r \neq r'$ since $q>p$. 
\begin{claim} 
\label{cl:umuc} 
For any $ r\in[p]$, if $q>p$ then
 \begin{align*}\label{eq:cube}
\bigsqcup_{\ell=0}^{a+m-2} \scC_{\ell,r}^m \,\mbox{ and }\,  \scQ\setminus \bigsqcup_{\ell=1}^{a+m-2}\scR_\ell \,\mbox{ are $\GL_m(\Z)$-equidecomposable},
\end{align*}
where 
\begin{align*}
\scQ=\scQ^{m}_r &:=\{x \in \R^m : 0 < x_1 < (a+m-1)q, \;0 < x_i < q, i \in [2,m]\} \cap H_{a+m-2,r}^m, \\
\scR_\ell=\scR_{\ell,r}^{m-1} &:= \{x \in \R^m : x_1 =\ell q, 0< x_i < q, i \in [2,m]\} \cap H_{a+m-2,r}^m.
\end{align*}
 Moreover, $\scQ$ is unimodularly equivalent to $(0,q)_\R^{m-1}$.
\end{claim}
\begin{proof}
By applying the translation $\tau:\R^m \to \R^m$ via $x \mapsto x-((a+m-2-\ell)q,0,\ldots,0)^T$, $\scC_{\ell,r}^m$ is unimodularly equivalent to
\begin{align*}
\{x \in \R^m : (a+m-2-\ell)q < x_1 < (a+m-1-\ell)q, 0< x_i < q, i \in [2,m]\} \cap H_{a+m-2,r}^m. 
\end{align*}
Obviously, $\scR_{\ell,r}^{m-1} \cap \scR_{\ell',r}^{m-1} =\emptyset$ for any $r$ if $\ell\ne\ell'$, and $\scR_{\ell,r}^{m-1}  \cap \scR_{\ell',r'}^{m-1} =\emptyset$ for any $\ell,\ell'$ if $r \neq r'$.
Hence $\bigsqcup_{\ell=0}^{a+m-2}\scC_{\ell,r}^m$ and $\scQ\setminus  \bigsqcup_{\ell=1}^{a+m-2}\scR_\ell$ are $\GL_m(\Z)$-equidecomposable, as desired. 

To prove the second statement, it suffices to prove that $[0,q]_\R^{m-1}$ is unimodularly equivalent to the closure $\overline{\scQ}$ of $\scQ$, 
$$\overline{\scQ}=\{x \in \R^m : 0 \le x_1 \le(a+m-1)q, \;0 \le x_i \le q, i \in [2,m]\} \cap H_{a+m-2,r}^m.
$$
The closure $\overline{\scQ}$ is a polytope whose  vertices are given by
\begin{align*}
((a+m-2-i)q+r,q^{(I)},0)^T ,((m-2-i)q+r,q^{(I)},q)^T,
\end{align*}
where $I \subseteq [2,m-1],$ $i=\#I$, and $q^{(I)}\in  \Z^{m-2}$ has the $j$-th entry equal to $q$ for $j \in I$ and $0$ otherwise. 
Note that $((a+m-2-i)q+r,q^{(I)},0)^T $ (resp. $((m-2-i)q+r,q^{(I)},q)^T$) comes from the intersection of 
the hyperplanes $x_m=0$ (resp. $x_m=q$), $x_j=q$ for $j \in I$, $x_{j'}=0$ for $j' \in [2,m-1]  \setminus I$ and $H_{a+m-2,r}^m$ for $I \subseteq [2,m-1]$. 
(Remark that $x_1=0$ does not intersect with $H_{a+m-2,r}^m$ within $\{x \in \R^m : 0 \leq x_1 \leq (a+m-1)q, \; 0 \leq x_i \leq q ,i \in [2,m]\}$, neither does $x_1=(a+m-1)q$ since $q>p$.)
One can show that the unimodular transformation $\sigma:\R^m \to \R^m$ via $x \mapsto 
\begin{pmatrix}
    1 & \cdots & 1& a  \\
      & \ddots &   &  \\
    &  & 1 &  \\
       & &  & 1 \\
       \end{pmatrix}x-(aq+r,0,\ldots,0)^T$ sends the polytope $\overline{\scQ}$ to a polytope whose vertices are given by
$$
((m-2)q,q^{(I)},0)^T,((m-2)q,q^{(I)},q)^T, \,\mbox{ where }\, I \subseteq [2,m-1].
$$
Clearly, this polytope is unimodularly equivalent to $[0,q]_\R^{m-1}$ which proves the claim.
\end{proof}

Use Eq. \eqref{eq:Ehrhart-def} again to write Eq. \eqref{eq:count} as
$$
 \chi^{\quasi}_\D(q) = {\rm L}_{(0,1)_\R^m}(q) -  \#\left( \bigsqcup_{r=1}^p \bigsqcup_{\ell=0}^{a+m-2}(\scC_{\ell,r}^m \cap \Z^m)  \right).
 $$
 Let us fix  $ r\in[p]$. 
By Claim \ref{cl:umuc},  we have
\begin{equation}
\label{eq:induction-r}
 \#\left(\bigsqcup_{\ell=0}^{a+m-2} (\scC_{\ell,r}^m \cap \Z^m)  \right) 
={\rm L}_{(0,1)_\R^{m-1}} (q)-
 \#\left(\bigsqcup_{\ell=1}^{a+m-2}( \scR_{\ell,r}^{m-1} \cap \Z^m)  \right). 
\end{equation}
It is not hard to see that 
$
\scR_{\ell,r}^{m-1}=H_{a+m-2-\ell,r}^{m-1} \cap (0,q)_\R^{m-1}.
$
Thus
\[ \#\left(\ \bigsqcup_{\ell=1}^{a+m-2}(\scR_{\ell,r}^{m-1}\cap \Z^m)  \right)=\#\left( \bigsqcup_{\ell=0}^{a+m-3}(H_{\ell,r}^{m-1} \cap (0,q)_\R^{m-1}\cap  \Z^{m-1}) \right).\]
Hence
\begin{align}\label{eq:induction}
 \#\left( \bigsqcup_{\ell=0}^{a+m-2} (\scC_{\ell,r}^m \cap \Z^m)  \right) 
={\rm L}_{(0,1)_\R^{m-1}} (q)-
 \#\left( \bigsqcup_{\ell=0}^{a+m-3} (\scC_{\ell,r}^{m-1}  \cap \Z^{m-1})  \right). 
\end{align}
Apply Eq. \eqref{eq:induction} repeatedly to have 
\begin{align*}
 \chi^{\quasi}_\D(q) = &{\rm L}_{(0,1)_\R^m}(q) -  \#\left( \bigsqcup_{r=1}^p  \bigsqcup_{\ell=0}^{a+m-2}(\scC_{\ell,r}^m \cap \Z^m)  \right)\\
= &{\rm L}_{(0,1)_\R^m}(q) - p \left( {\rm L}_{(0,1)_\R^{m-1}}(q)- {\rm L}_{(0,1)_\R^{m-2}}(q)+\cdots+(-1)^{m-3}{\rm L}_{(0,1)_\R^2}(q)\right) \\
+&(-1)^{m-1}\left(p{\rm L}_{(0,1)_\R^1}(q)- \#\left(\bigsqcup_{r=1}^p \bigsqcup_{\ell=1}^a(\scR_{\ell,r}^{1}\cap \Z^2)\right) \right). 
\end{align*}
(Note that we have applied Eq. \eqref{eq:induction-r} in the very last step.)
This completes the proof of Eq. \eqref{eq:s=1}.

It is easily seen that 
$$
 \bigsqcup_{\ell=1}^a(\scR_{\ell,r}^{1}\cap \Z^2) = \left\{ \ell \in [a] : q-\frac{\ell q-r}{a} \in [q-1]\right\}.
$$
If $a=1$, then obviously, 
\begin{align*}
c:= \#\left(\bigsqcup_{r=1}^p\left\{ \ell \in [a]: q-\frac{\ell q-r}{a} \in [q-1]\right\}  \right)
&=  \# \left\{r\in [p] : r \in [q-1] \right\}   = p. 
\end{align*}
If $q\ge 2p$ and $a=p$,  then 
\begin{align*}
c=  \#\left(\bigsqcup_{r=1}^p\left\{ \ell \in [p] :  q-\frac{\ell q-r}{p} \in \Z\right\}  \right) 
=  \#\left(\bigsqcup_{r=1}^p\left\{  \ell \in [p] : \ell q \equiv r \bmod{p}\right\}  \right) = p. 
\end{align*}
This completes the proof of Eq. \eqref{eq:s=1-inp}.
\end{proof}

\begin{example}
\label{ex:PC-q=4}  
We give an illustration of Lemma \ref{lem:s=1} in a particular case when $m=p=a=2$ and $q=4$ (see Figure \ref{fig:PC}). 
By Eq. \eqref{eq:count},  $\chi^{\quasi}_\D(q)$ is given by the number of lattice points in the open cube $(0,4)^2_\R$ but not in any hyperplane $H_{\ell,r} : x+2y=4\ell+r$ for $\ell=0,1 ,2$, $r=1,2$. 
It is easy to see that there are such $5$ lattice points highlighted in  Figure \ref{fig:PC}. 
To count the lattice points in  $\scC_{\ell,r} =H_{\ell,r} \cap (0,4)^2_\R$, following Claim \ref{cl:umuc}, we translate these open segments by integral vectors to form the sets $\scQ_1  \setminus (\scR_{1,1}\cup \scR_{2,1})$, $\scQ_2  \setminus (\scR_{1,2} \cup\scR_{2,2})$, where 
\begin{align*}
\scQ_1 & =  \mathrm{conv} \{(1,4)^T, (9,0)^T\} \setminus \{(1,4)^T, (9,0)^T\}, \\
\scQ_2 & =  \mathrm{conv} \{(2,4)^T, (10,0)^T\} \setminus \{(2,4)^T, (10,0)^T\}, 
\end{align*}
and $\scR_{1,1} =\{ (4, 5/2)^T\}$,  $\scR_{2,1} =\{  (8, 1/2)^T\}$,  $\scR_{1,2} =\{  (4, 3)^T\}$,  $\scR_{2,2} =\{ (8, 1)^T\}$. Note that both $\scQ_1$ and $\scQ_2$ are  unimodularly equivalent to $(0,4)_\R =   \mathrm{conv} \{(0,4)^T, (0,0)^T\}  \setminus \{(0,4)^T, (0,0)^T\} $ under the unimodular transformations $x \mapsto 
\begin{pmatrix}
    1 &2  \\
0& 1 \\
       \end{pmatrix}x-(9,0)^T$, $x \mapsto 
\begin{pmatrix}
    1 &2  \\
0& 1 \\
       \end{pmatrix}x-(10,0)^T$, respectively. 
Moreover, among the sets $\scR_{\ell,r}$'s, only $\scR_{1,2}$ and $\scR_{2,2}$ contain integral vectors. Thus,
$${\rm L}_{(0,1)_\R^2}(4) -2{\rm L}_{(0,1)_\R}(4) +\#\left(\bigsqcup_{r=1}^2 \bigsqcup_{\ell=1}^2(\scR_{\ell,r}\cap \Z^2)\right)  = 9 -2\cdot3+2=5,$$
which is compatible with the calculation above.
\begin{figure}[!ht]
\centering
    \includegraphics[width=14cm,height=6cm]{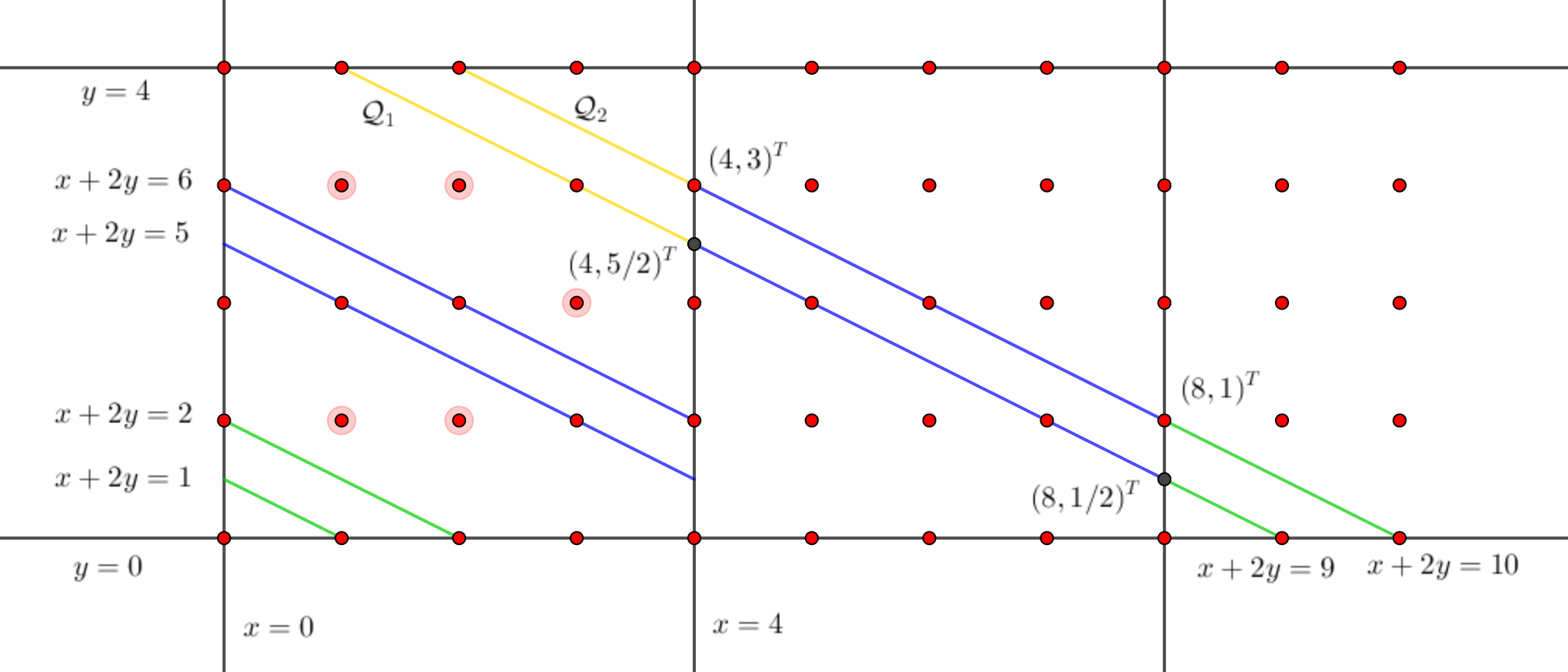}
    \caption{An illustration of Lemma \ref{lem:s=1} when $m=p=a=2$ and $q=4$.}
    \label{fig:PC}
\end{figure}
\end{example}

We are now ready to  prove Theorem \ref{thm:main-Ehrhart}.
\begin{proof}[\textbf{Proof of  Theorem \ref{thm:main-Ehrhart}}]
By Eq. \eqref{eq:Ehrhart-def},  $\chi^{\quasi}_\A(q)$ is given by 
\begin{align*}
\#\{x \in [q-1]^m :   x_m \neq kq/g, k \in [g-1], x_1+\cdots+x_{m-1}+px_m \neq \ell q+r, \ell \in \Z, r\in[p]\}, 
\end{align*}
where $g=\gcd(q,s)$. 
Suppose $q\ge 2p$. 
Thus if we set $a=p$ in Eq. \eqref{eq:count}, then
\begin{equation}\label{eq:Nk}
\chi^{\quasi}_\A(q) = RHS \,\eqref{eq:count} -  \bigsqcup_{k=1}^{g-1}\scN_k,
\end{equation}
where 
 \begin{align*}
\scN_k:= & \, \# \{x \in [q-1]^m :   x_m= kq/g,  x_1+\cdots+x_{m-1}+px_m \neq \ell q+r, \ell \in \Z, r\in[p]\}   \\
= & \,\#  \{x \in [q-1]^{m-1}:   x_1+\cdots+x_{m-1}  \neq \ell q+r, \ell \in \Z, r\in[p]\}  . 
\end{align*}
The formula of $\scN_k$ above is essentially Eq. \eqref{eq:count} where $a$ and $m$ are replaced by $1$ and $m-1$, respectively. 
Thus by  Eq. \eqref{eq:s=1-inp}, for each $k \in [g-1]$,
\begin{equation}\label{eq:each-Nk}
\scN_k=
{\rm L}_{(0,1)_\R^{m-1}}(q) + p\sum_{k=1}^{m-1}(-1)^k{\rm L}_{(0,1)_\R^{m-1-k}}(q).
\end{equation}

Combining Eq. \eqref{eq:Nk} and  Eq. \eqref{eq:each-Nk}, we have
$$
\chi^{\quasi}_\A(q) ={\rm L}_{(0,1)_\R^{m}}(q) + p\sum_{k=1}^{m}(-1)^k{\rm L}_{(0,1)_\R^{m-k}}(q)-  (g-1)\left({\rm L}_{(0,1)_\R^{m-1}}(q) + p\sum_{k=1}^{m-1}(-1)^k{\rm L}_{(0,1)_\R^{m-1-k}}(q)\right),
$$
which agrees with RHS \eqref{eq:Ehrhart}, as desired.
 \end{proof}

An immediate consequence of Theorem \ref{thm:main-Ehrhart} is the following  reciprocity property.
\begin{corollary} 
\label{thm:cor-main-Ehrhart} 
The following  holds
\begin{equation*}
\label{eq:cor-Ehrhart}
(-1)^m\chi^{\quasi}_\A(-q)=(q+\gcd(q,s))\left({\rm L}_{[0,1]_\R^{m-1}}(q)+p\sum_{k=1}^{m-1}{\rm L}_{[0,1]_\R^{m-1-k}}(q)\right) + p. 
\end{equation*}
\end{corollary}

 \section{Period collapse from root systems}
\label{sec:rs} 
Let   $V$ be an $m$-dimensional Euclidean vector space with the standard inner product $(\cdot,\cdot)$.
 Let $\Phi$ be an irreducible (crystallographic) root system in $V$, with a fixed positive system $\Phi^+ \subseteq \Phi$ and the associated set of simple roots $\Delta := \{\alpha_1,\ldots,\alpha_m \}$. 
  The highest root $\tilde{\alpha} \in \Phi^+$ can be written uniquely as a linear combination of the simple roots $\tilde{\alpha}= \sum_{i=1}^m c_i\alpha_i$  with all $c_i \in  \Z_{>0}$. 
Let $\{\varpi^\vee_1, \ldots ,\varpi^\vee_m\}$ be the dual basis of $\Delta$, namely, $(\alpha_i ,\varpi^\vee_j) = \delta_{ij}$. 
The \emph{coweight lattice} of $\Phi$ is defined by $Z(\Phi^\vee):=\bigoplus_{i=1}^m \Z\varpi^\vee_i$.  
The fundamental domain $P^\diamondsuit = \sum_{i=1}^m (0,1]_\R \varpi^\vee_i$ of the coweight lattice is called the \emph{fundamental parallelepiped}.

For $k \in \Z$ and $\alpha \in  \Phi$, the \emph{affine hyperplane} $\widetilde{H}_{\alpha,k} $ is defined by
$\widetilde{H}_{\alpha,k} :=\{x\in V:(\alpha,x)=k\}$. 
A connected component of $V \setminus \bigcup_{ \alpha\in \Phi^+,k \in \Z} \widetilde{H}_{\alpha,k}$ is called an \emph{alcove}. 
Let $A^\circ$ be the \emph{fundamental alcove} of $\Phi$, then its closure $\overline{A^\circ}=\mathrm{conv} \left\{0, \frac{\varpi^\vee_1}{c_1},\ldots, \frac{\varpi^\vee_m}{c_m}\right\} \subseteq  \overline{P^\diamondsuit}$ can be regarded as a rational polytope in $\bigoplus_{i=1}^m \R\varpi^\vee_i$.
Let $W$ be the \emph{Weyl group} and let $Q(\Phi^\vee)$ be the \emph{coroot lattice} of $\Phi$.
The \emph{affine} Weyl group $W_{\rm aff} := W \ltimes Q(\Phi^\vee)$ acts simply transitively on the set of alcoves and admits $\overline{A^\circ}$ as a fundamental domain for its action on $V$.

Let $A$ be an alcove.\footnote{This section is mostly independent of the previous sections, we hope the reader will have no confusion between the notation $A$ for alcoves here and the earlier notation for matrices.}
A \emph{wall} of $A$ is a hyperplane that supports a facet of $A$. 
The \emph{ceilings} of $A$ are the walls which do not pass through the origin and have the origin on the same side as $A$. 
The  \emph{upper closure} $A^\diamondsuit$  of $A$ is the union of $A$ and its facets supported by the ceilings of $A$.
Thus, 
$$P^\diamondsuit=\bigsqcup_{A:\, \text{alcove},\,A \subseteq P^\diamondsuit}A^\diamondsuit,
$$ 
which is known as the \emph{Worpitzky partition}, e.g., \cite[Proposition 2.5]{Y18W}, \cite[Exercise 4.3]{H90}.

 \begin{definition}{\cite[Definition 4.8]{ATY20}}
 \label{def:compatibleW}   
A subset $\Psi \subseteq \Phi^+$ is said to be \emph{Worpitzky-compatible} (or \emph{compatible} for short) if for each alcove $A\subseteq P^\diamondsuit$, the intersection $A^\diamondsuit \cap \widetilde{H}_{\alpha, k_\alpha}$ for $\alpha\in \Psi,k_\alpha \in \Z$ is either empty, or contained in a ceiling $\widetilde{H}_{\beta, k_\beta}$ of $A$ for $\beta\in \Psi,k_\beta \in \Z$. 
\end{definition}

A subset $\Psi \subseteq \Phi^+$ is said to be \emph{coclosed} if for any $\alpha \in \Psi$ and  for $\beta_1, \beta_2 \in \Phi^+$  such that $\alpha=\beta_1+ \beta_2$, either $\beta_1\in \Psi$ or $\beta_2\in \Psi$. 
The coclosed subsets are called \emph{strongly Worpitzky-compatible}  in \cite[Definition 2.1]{TT21}, and the equivalence between these two concepts is not hard to verify, e.g., one can use  \cite[Lemma 3.2]{S05}. 
If both $\beta_1, \beta_2$ in the definition of the coclosed subsets are in $\Phi^+$, then the subset  $\Psi$ is called an \emph{ideal} of $\Phi^+$.\footnote{A more common definition of the ideals is related to the partial order $\ge$ on $\Phi^+$: $\beta_1 \ge \beta_2$ if $\beta_1-\beta_2 \in\sum_{i=1}^m \Z_{\ge 0}\alpha_i$; in this setting, a subset $\Psi\subseteq\Phi^+$ is an ideal of $\Phi^+$  if for $\beta_1,\beta_2 \in \Phi^+$, $\beta_1 \ge \beta_2, \beta_ 1 \in \Psi$ implies $\beta_2 \in \Psi$. By the similar method, one also can show that these two definitions are equivalent.}
It is proved that coclosed subsets are compatible \cite[Proof of Theorem 4.16]{ATY20} (see also  \cite[Theorem 2.3]{TT21}).
In particular, if $\delta \in \Phi^+$ then $\Phi^+ \setminus \{\delta\}$ is coclosed (but not necessarily an ideal), hence compatible.
 
Let $\Psi \subseteq \Phi^+$, and set $\Psi^c:=\Phi^+ \setminus \Psi$. 
For $\delta \in \Phi^+$, we simply write $\delta^c$ to denote $\{\delta\}^c$. 
Now we recall the notion of \emph{descent} and \emph{ascent} statistics following \cite[\S 4 and \S 5]{ATY20}. 
Note that  $\tilde{\alpha}= \sum_{i=1}^m c_i\alpha_i$ denotes the highest root of $\Phi$ and we also denote $\alpha_0 := -\tilde{\alpha}$, $c_0 := 1$.
For $w\in W$, define
$$
\dsc_\Psi(w) := \sum_{0 \le i \le m,\, w(\alpha_i)\in -\Psi^c}c_i\;\;\text{ and }\;\;
\asc_\Psi(w):=  \sum_{0 \le i \le m, \, w(\alpha_i)\in \Psi^c}c_i.$$
By \cite[Lemma 4.5]{ATY20}, if $w_1, w_2 \in W$ and $w_1( A^\circ)=w_2( A^\circ) + \gamma$ for some $\gamma \in V$, then $\dsc_\Psi(w_1)=\dsc_\Psi(w_2)$. 
Similarly, $\asc_\Psi(w_1)=\asc_\Psi(w_2)$. 
Thus, we can extend $\dsc_\Psi$ and $\asc_\Psi$ to functions on the set of all alcoves as follows:
 \begin{definition}\label{def:asc-ext}
 Let $A'$ be an arbitrary alcove. 
 Since $W_{\rm aff}= W \ltimes Q(\Phi^\vee)$ acts simply transitively on the set of alcoves, we can write $A'=w( A^\circ) + \gamma$ for some $w\in W$ and $\gamma \in Q(\Phi^\vee)$.   
Define
$$\dsc_\Psi(A'):=\dsc_\Psi(w) \;\;\text{ and }\;\;\asc_\Psi(A'):=\asc_\Psi(w).$$
\end{definition}

Let $a\le b$ be integers. 
Following \cite[\S 5]{ATY20}, we define the \emph{deformed Weyl arrangements of type I} of $\Psi$  by 
$$\mathrm{I}_\Psi^{[a,b]} := \{\widetilde{H}_{\alpha,k} :\alpha \in \Psi, k \in [a,b]\}.$$ 
In particular, for $k \in \Z_{\ge0}$, the \emph{extended Shi arrangement} of $\Psi$ is defined by
$$\mathrm{Shi}_\Psi^{[1-k,k]} := \mathrm{I}_\Psi^{[1-k,k]}.$$ 

Let $C_\Psi$ be the \emph{coefficient matrix} of $\Psi$ with respect to $\Delta$, i.e., $C_\Psi = (C_{ij})$ is the $m \times \#\Psi$ integral matrix that satisfies 
$$\Psi = \left\{\sum_{i=1}^m C_{ij}\alpha_i : 1 \le j \le \#\Psi  \right\}.$$
\begin{example}
\label{ex:B2G2}   
Let $\Phi=B_2$ with $\Phi^+= \{\alpha_1,\alpha_2,  \alpha_1+\alpha_2, 2\alpha_1+\alpha_2\}$ where $\Delta=\{\alpha_1,\alpha_2\}$ and $\alpha_1$ is the simple short root. The coefficient matrix of  $\Phi^+$ w.r.t. $\Delta$ is given by 
$$
C_{\Phi^+} = 
\begin{pmatrix}
1 & 0 & 1 &   2 \\
0 & 1 & 1 &   1 
\end{pmatrix}.
$$
 Let $\Phi=G_2$ with $\Phi^+= \{\alpha_1,\alpha_2,  \alpha_1+\alpha_2, 2\alpha_1+\alpha_2,3\alpha_1+\alpha_2,3\alpha_1+2\alpha_2\}$ where $\Delta=\{\alpha_1,\alpha_2\}$ and $\alpha_1$ is the simple short root. The coefficient matrix of  $\Phi^+$ w.r.t. $\Delta$ is given by 
$$
C_{\Phi^+} = 
\begin{pmatrix}
1 & 0 & 1 &   2& 3 &  3 \\
0 & 1 & 1 &   1 & 1 &   2
\end{pmatrix}.
$$
\end{example}
\noindent
Let $A_\Psi$ be an $m \times (2k\cdot\#\Psi)$ integral matrix whose columns are given by $(C_{1j},\ldots,C_{mj},\ell)^T$ where $\ell \in [1-k,k], 1 \le j \le \#\Psi$. 
Thus  the extended Shi arrangement  $\mathrm{Shi}_\Psi^{[1-k,k]}$ is the non-central arrangement defined by $A_\Psi$ in the sense of  \S\ref{sec:CQP}. 
Denote by $\chi^{\quasi}_{\mathrm{Shi}_{\Psi}^{[1-k,k]}}(q)$ the corresponding characteristic quasi-polynomial. 
Let ${\rm L}_{\overline{A^\circ}}(q) = \#(q\overline{A^\circ} \cap Z(\Phi^\vee))$ be the Ehrhart quasi-polynomial of $\overline{A^\circ} $ w.r.t. the  coweight lattice. 
Let $\h= \sum_{i=0}^m c_i$ denote the \emph{Coxeter number} of $\Phi$.
\begin{theorem}{\cite[Theorem 5.1]{Y18W}}
\label{thm:yoshi} 
If $k \in \Z_{\ge0}$ then 
\begin{align*}
\chi^{\quasi}_{\mathrm{Shi}_{\Phi^+}^{[1-k,k]}}(q)  
& =\sum_{A':\, \text{alcove},\,A' \subseteq P^\diamondsuit} {\rm L}_{\overline{A^\circ}}(q-k\h-{\asc}_{\emptyset}(A' )) \\
 & =  (q-k\h)^m.
\end{align*}
\end{theorem}
\noindent
Let  $\rho_{\Phi^+}$ denote the lcm period of $C_{\Phi^+}$. 
It is known that $\rho_{\Phi^+}\ge2$ in all cases except in type $A$ when it equals $1$, e.g., \cite[Remark 3.3]{KTT10}. 
Thus the formula in Theorem \ref{thm:yoshi} implies that period collapse occurs in these cases.

For $\delta \in \Phi^+$, the sets $\delta^c$ and $\Phi^+$ share many important common properties such as the corresponding Weyl arrangements are free in the sense of Terao, e.g., \cite{OST87}, or by the previous discussion, both sets are coclosed hence compatible.
We wish to study $\delta^c$ through the aspect of period collapse as well and we give below a similar formula of $\chi^{\quasi}_{\mathrm{Shi}_{\delta^c}^{[1-k,k]}}(q)$.
\begin{proposition}
\label{prop:CQP-I}
If $k \in \Z_{\ge0}$ and $\delta \in \Phi^+$, then 
$$
\chi^{\quasi}_{\mathrm{Shi}_{\delta^c}^{[1-k,k]}}(q) = \sum_{A':\, \text{alcove},\,A' \subseteq P^\diamondsuit} {\rm L}_{\overline{A^\circ}}(q-k\h-{\asc}_{\emptyset}(A' )+k{\asc}_{\delta^c}(A' ) +k{\dsc}_{\delta^c}(A' )).
$$
\end{proposition} 

\begin{proof} 
Since $\delta^c$ is compatible, we can apply \cite[Theorem 5.6(i)]{ATY20}.
\end{proof}

Proposition \ref{prop:CQP-I} is in general hard to apply, however, in particular cases it enables us to produce some examples for period collapse.
\begin{example}
\label{ex:B2}   
Let $\Phi=B_2$ as in Example \ref{ex:B2G2}. Note that 
$$
{\rm L}_{\overline{A^\circ}}(q)
= \begin{cases}
\frac{1}{4}(q+1)(q+3) & \mbox{if $\gcd(q,2)=1$}, \\
\frac{1}{4}(q+2)^2 & \mbox{if $\gcd(q,2)=2$}.
\end{cases}
$$
\begin{enumerate}[(i)]
\item If $\delta$ is a short root, then period collapse occurs in  $\chi^{\quasi}_{\mathrm{Shi}_{\delta^c}^{[1-k,k]}}(q)$. More precisely, the quasi-polynomial has lcm period $2$ but minimum period $1$,
\begin{align*}
\chi^{\quasi}_{\mathrm{Shi}_{\delta^c}^{[1-k,k]}}(q)  
& = {\rm L}_{\overline{A^\circ}}(q-4k-3)+2  {\rm L}_{\overline{A^\circ}}(q-2k-2)+ {\rm L}_{\overline{A^\circ}}(q-4k-1) \\
& = {\rm L}_{\overline{A^\circ}}(q-2k-3)+2  {\rm L}_{\overline{A^\circ}}(q-4k-2)+ {\rm L}_{\overline{A^\circ}}(q-2k-1) \\
 & =  q^2-6kq+10k^2.
\end{align*}

\item However, if $\delta$ is a long root, then period collapse \textbf{does not} occur in  $\chi^{\quasi}_{\mathrm{Shi}_{\delta^c}^{[1-k,k]}}(q)$. It is because the quasi-polynomial has lcm period $1$ hence minimum period $1$. 
More precisely, 
\begin{align*}
\chi^{\quasi}_{\mathrm{Shi}_{\delta^c}^{[1-k,k]}}(q)  
=  q^2-6kq+9k^2.
\end{align*}
\end{enumerate}
\end{example}

\begin{example}
\label{ex:G2}   
Let $\Phi=G_2$ as in Example \ref{ex:B2G2}. Note that 
$$
{\rm L}_{\overline{A^\circ}}(q)
= \begin{cases}
\frac{1}{12}(q+1)(q+5) & \mbox{if $\gcd(q,6)=1$}, \\
\frac{1}{12}(q+2)(q+4) & \mbox{if $\gcd(q,6)=2$}, \\
\frac{1}{12}(q+3)^2 & \mbox{if $\gcd(q,6)=3$}, \\
\frac{1}{12}(q^2+6q+12)  & \mbox{if $\gcd(q,6)=6$}.
\end{cases}
$$
\begin{enumerate}[(i)]
\item If $\delta$ is a short root, then period collapse occurs in  $\chi^{\quasi}_{\mathrm{Shi}_{\delta^c}^{[1-k,k]}}(q)$. More precisely, the quasi-polynomial has lcm period $6$ but minimum period $1$,
$$\chi^{\quasi}_{\mathrm{Shi}_{\delta^c}^{[1-k,k]}}(q) = q^2-10kq+27k^2 .$$
\item When $\delta$ is a long root, it becomes more complicated to compute $\chi^{\quasi}_{\mathrm{Shi}_{\delta^c}^{[1-k,k]}}(q)$ as the calculation depends on the parameter $k$ as well. 
But we can still find some case where period collapse occurs. 
For example, when $\delta=\alpha_2$, the quasi-polynomial has lcm period $6$. 
If $k \equiv 3 \bmod 6$, then $\chi^{\quasi}_{\mathrm{Shi}_{\delta^c}^{[1-k,k]}}(q)$ has minimum period $1$,
$$\chi^{\quasi}_{\mathrm{Shi}_{\delta^c}^{[1-k,k]}}(q) = q^2-10kq+\frac{77k^2}3 .$$
If $k \equiv 1 \bmod 6$, then $\chi^{\quasi}_{\mathrm{Shi}_{\delta^c}^{[1-k,k]}}(q)$ has minimum period $3$,
$$\chi^{\quasi}_{\mathrm{Shi}_{\delta^c}^{[1-k,k]}}(q)
=\begin{cases}
q^2-10kq+\frac{77k^2+1}3& \mbox{if $\gcd(q,3)=1$}, \\
q^2-10kq+\frac{77k^2-2}3& \mbox{if $\gcd(q,3)=3$}.
\end{cases}$$
\end{enumerate}
\end{example}

Based on the calculation above, we pose the following conjecture.
\begin{conjecture}
\label{conj:-delta}
Let  $k \in \Z_{\ge0}$.
If $\delta \in \Phi^+$, then $\chi^{\quasi}_{\mathrm{Shi}_{\delta^c}^{[1-k,k]}}(q)$ either has minimum period $1$ (i.e., is a polynomial), or results in period collapse.
 \end{conjecture}
 
 One can also formulate and investigate a version of Conjecture \ref{conj:-delta} for coclosed subsets.
 We close this section by addressing a remark on the extended Linial arrangements.
 
 \begin{remark}
\label{rem:Linial}
For $n \in \Z_{\ge0}$, the \emph{extended Linial arrangement} of $\Psi \subseteq \Phi^+$ is defined by
$$\mathrm{Lin}_\Psi^{[1,n]} := \mathrm{I}_\Psi^{[1,n]} .$$ 
It is recently shown that $\gcd(n+1,\rho_{\Phi^+})$ is a period of $\chi^{\quasi}_{\mathrm{Lin}_{\Phi^+}^{[1,n]}}(q)$ \cite[Theorem 3.1]{Tam20} yielding period collapse in certain cases, e.g., when $\rho_{\Phi^+}\ge2$ and $\gcd(n+1,\rho_{\Phi^+})=1$. 
For arbitrary $n$, however, $\chi^{\quasi}_{\mathrm{Lin}_{\Phi^+}^{[1,n]}}(q)$ neither has minimum period $1$ nor results in period collapse. 
For example, let $\Phi=B_2$ as in Example \ref{ex:B2G2}.
If $n$ is odd, then period collapse does not occur in $\chi^{\quasi}_{\mathrm{Lin}_{\Phi^+}^{[1,n]}}(q)$ because by \cite[Theorem 5.2]{Y18W}, 
$$\chi^{\quasi}_{\mathrm{Lin}_{\Phi^+}^{[1,n]}}(q)
=\begin{cases}
q^2-4kq+\frac{9n^2+2n-1}2& \mbox{if $\gcd(q,2)=1$}, \\
q^2-4kq+\frac{9n^2+2n+1}2& \mbox{if $\gcd(q,2)=2$}.
\end{cases}$$
Let $\delta \in \Phi^+$. 
By  \cite[Theorem 5.6(ii)]{ATY20}, 
$$
\chi^{\quasi}_{\mathrm{Lin}_{\delta^c}^{[1,n]}}(q)= \sum_{i \in [N]} {\rm L}_{\overline{A^\circ}}(q-(n+1){\asc}_{\emptyset}(A_i^\circ)+n{\asc}_{\delta^c}(A_i^\circ)).
$$
The same conclusion applies to $\chi^{\quasi}_{\mathrm{Lin}_{\delta^c}^{[1,n]}}(q)$. 
For example, if  $\Phi=B_2$, $\delta=\alpha_1$ (short) and $n$ is odd, then period collapse does not occur in $\chi^{\quasi}_{\mathrm{Lin}_{\delta^c}^{[1,n]}}(q)$ and
$$\chi^{\quasi}_{\mathrm{Lin}_{\delta^c}^{[1,n]}}(q)
=\begin{cases}
q^2-3kq+\frac{5n^2-1}2& \mbox{if $\gcd(q,2)=1$}, \\
q^2-3kq+\frac{5n^2+1}2& \mbox{if $\gcd(q,2)=2$}.
\end{cases}$$

\end{remark}

 \section{Proof of the second main result: the central case}
\label{sec:open} 
In this section, we prove the second main result (Theorem~\ref{thm:central}) of the paper.
Our proof relies on the recent development of the characteristic quasi-polynomials of central arrangements via the theory of toric arrangements  \cite{TY19}. 
We will use in the proof some notations from \S \ref{sec:CQP}.

\begin{proof}[\textbf{Proof of  Theorem \ref{thm:central}}]
Let $C=(c_1,\ldots,c_n) \in \Mat_{m\times n}(\Z)$ with $c_j =(c_{1j},\ldots,c_{mj})^T \ne (0,\ldots,0)^T$ ($1\le j\le n$).
Let $\TT=(\mathbb{S}^1)^m$ be the $m$-torus. 
The matrix $C$ defines a toric arrangement $\A := \{ T_j : 1\le j\le n\}$ in $\TT$, where
\[
T_j:=\left\{(t_1, \dots, t_m)\in \TT : \prod_{i=1}^mt_i^{c_{ij}}=1\right\}. 
\]
For $\emptyset \ne J\subseteq [n]$, denote $T_J:=\bigcap_{j\in J}T_j$, and denote $\TT:=T_{\emptyset}$.
Let $\emptyset \ne J\subseteq [n]$.
Note that $C_J$ determines a homomorphism $\TT\to (\mathbb{S}^1)^{\#J}$ 
whose kernel is $T_J$. 
There exist $U_J\in \GL_m(\Z)$ and $V_J\in \GL_{\#J}(\Z)$ such that 
\[
U_J C_J V_J=\diag(e_{J,1}, \dots, e_{J, \ell(J)}, 0, \dots, 0). 
\]
Hence $T_J$ is isomorphic to 
$$
\left\{(t_1, \dots, t_m)\in (\mathbb{S}^1)^m: t_1^{e_{J, 1}}=1, \dots, t_{\ell(J)}^{e_{J, \ell(J)}}=1\right\}. 
$$
Therefore, $T_J$ is a disjoint union of $\prod_{i=1}^{\ell(J)} e_{J, i}$ tori of dimension $m-\ell(J)$. 
Let $Z$ be a connected component of $T_J$. Then there exist $\omega_1, \dots, 
\omega_{\ell(J)}\in \mathbb{S}^1$ such that $\omega_i^{e_{J, i}}=1$ for $ 1\le i \le \ell(J)$ and 
\begin{equation}
\label{eq:comp}
Z=\left\{(\omega_1, \dots, \omega_{\ell(J)}, t_{\ell(J)+1}, \dots, t_m): t_j\in \mathbb{S}^1, \ell(J)+1 \le j \le m\right\}
\simeq (\mathbb{S}^1)^{m-\ell(J)}. 
\end{equation}

An element $z\in \TT$ is called a $k$-torsion element $(k\ge1)$ if $z^k=1$, and an element in $\TT$ is called a torsion element if it is a $k$-torsion element for some $k\ge1$. 
Let $Z\subseteq \TT$ be a subset and let $\tor(Z)$ denote the set of  torsion elements of $Z$. 
Denote by $\tau(Z)$ the minimum of the orders of the torsion elements in $Z$, that is, 
\[
\tau(Z):=
\begin{cases}
\min\{\ord(z):  z\in \tor(Z)\} & \mbox{if $\tor(Z) \ne\emptyset$,}\\
\infty & \mbox{otherwise}. 
\end{cases}
\]
Let $Z$ be a connected component of $T_J$ as above. 
On the one hand, the number $\tau(Z)$ does not depend on the isomorphism because a group isomorphism preserves the torsion subgroup and the orders of elements. 
On the other hand, to compute $\tau(Z)$ by Eq. \eqref{eq:comp}, it is enough to look at the elements $z$'s in $Z$ with $t_j=1$ for all $\ell(J)+1 \le j \le m$. Thus
\begin{equation}
\label{eq:tau}
\tau(Z)=\lcm\left(\ord(\omega_i): 1\le i \le \ell(J)\right). 
\end{equation}

\begin{lemma}
\label{lem:z}
Let $\emptyset \ne J\subseteq [n]$ and let $Z$ be a connected component of $T_J$. 
\begin{itemize}
\item[(a)] 
If $z\in \tor(Z)$, then $\tau(Z)\mid\ord(z)$. 
\item[(b)] 
$Z$ contains a $k$-torsion element if and only if $\tau(Z)\mid k$. 
\end{itemize}
\end{lemma}
\begin{proof}
Straightforward from Eq.  
(\ref{eq:comp}) and Eq. (\ref{eq:tau}). 
\end{proof}

\begin{lemma}
\label{lem:max}
If $\emptyset \ne J\subseteq [n]$, then $\tau(Z) \mid e_{J, \ell(J)}$ for all components $Z$'s, and $\tau(Z)=e_{J, \ell(J)}$ 
for some component $Z$ of $T_J$. 
 \end{lemma}
\begin{proof}
Straightforward Eq.  
(\ref{eq:comp}) and Eq. (\ref{eq:tau}). 
\end{proof}
 
Let $\chi^{\quasi}_C(q)= q^m + d_{m-1}(q)q^{m-1} + \dots +d_0(q) $ denote the  characteristic quasi-polynomial of the matrix $\begin{pmatrix} C\\ 0\end{pmatrix}$. 
Let $\rho_C$ be the lcm period, and let $f^k_C(t)$ ($1 \le k \le \rho_C$) denote the $k$-constituent of $\chi^{\quasi}_C(q)$. 
Thus
\begin{equation}
\label{eq:k-const} 
f^k_C(t)= t^m + d_{m-1}(k)t^{m-1} + \dots +d_0(k).
\end{equation}
Let us recall the description of the constituent $f^k_C(t)$ via combinatorics of the toric arrangement $\A$ following \cite[\S4.2]{TY19}. 
Let 
\[
L:=\{Z: J\subseteq [n], \mbox{ $Z$ is a connected component of $T_J$}\}
\]
be the set of the connected components of the intersections of the tori in $\A$. 
The set $L$ can be naturally turned into a poset whose partial order is given by reverse inclusion and $\TT=T_{\emptyset}$ is the unique minimal element in $L$. 
For $0 \le r \le m$,  let $ L^r:= \{ Z \in L : \dim Z=r\}\subseteq L$. For $k\ge1$, define
\[
L[k]:=\{Z\in L : \mbox{$Z$ contains a $k$-torsion element}\}. 
\]
Observe that $L[k]$ is an order ideal of $L$ in the sense that if $Y$ covers $Z$ in $L$ and $Y\in L[k]$ then $Z\in L[k]$. 
Let $ L^r[k]:= \{ Z \in  L[k] : \dim Z=r\}\subseteq  L[k]$.
By \cite[Theorem 4.7]{TY19}, the constituent  $f^k_C(t)$ coincides with the characteristic polynomial of $L[k]$, 
namely, 
\begin{equation*}
f^k_C(t)=\sum_{Z\in L[k]}\mu(\TT, Z)t^{\dim Z}, 
\end{equation*}
where $\mu: L \times L \to \Z$ is the M\"obius function. 
It follows from Eq. \eqref{eq:k-const}  that if $0 \le r \le m-1$ then
\begin{equation*}
d_r(k)=\sum_{Z\in L^r[k]}\mu(\TT, Z). 
\end{equation*}
By Lemma \ref{lem:z}(b), we can rewrite the equality above as
\begin{equation}
\label{eq:coe} 
(-1)^{m-r} d_r(k)=\sum_{Z\in  L^r,\, \tau(Z)\mid k}(-1)^{m-r} \mu(\TT, Z). 
\end{equation}

Now apply Lemma \ref{lem:max}, we have
$
\rho_C=\lcm\left(\tau(Z):Z\in L\setminus\{\TT\}\right). 
$
For $0 \le r \le m-1$, define 
\begin{equation*}
\rho_{C, r}:=\lcm\left(\tau(Z):Z\in  L^r\right). 
\end{equation*}
Clearly, 
\begin{equation}
\label{eq:rhoC} 
\rho_C=\lcm\left(\rho_{C, r}:0 \le r \le m-1\right). 
\end{equation}

\begin{claim} 
\label{cl:min} 
For each $0 \le r \le m-1$, the minimum period of the coefficient $d_r(q)$ is $\rho_{C, r}$. 
\end{claim}
\begin{proof}
By Theorem \ref{thm:KTT}, $\lcm (e_{J,\ell(J)}: \emptyset \ne J \subseteq [n], \ell(J)=r)$ is  a period of $d_r(q)$ since $\tilde{d}_J(q)$ has the minimum period $e_{J, \ell(J)}$ (see also \cite[Lemma 2.1]{KTT08}). 
Since $\rho_{C, r}$ essentially equals this number, it is a period of $d_r(q)$.
Let $\rho_0$ be the minimum period of $d_r(q)$. 
Thus $\rho_0 \mid \rho_{C, r}$ and use the fact that $\rho_0$ is a period to have $d_r(\rho_0) = d_r(\rho_{C, r})$. 
Note that by definition, $\tau(Z)\mid \rho_{C, r}$ for all $Z\in  L^r$.
Now apply Eq. \eqref{eq:coe}, we have
$$
\sum_{Z\in  L^r,\, \tau(Z)\mid \rho_0}(-1)^{m-r} \mu(\TT, Z)=\sum_{Z\in  L^r}(-1)^{m-r} \mu(\TT, Z). 
$$
Every term $(-1)^{m-r} \mu(\TT, Z)$ is positive (see e.g.,  \cite[Corollary 3.6]{TY19}). 
The above equality happens only when $\tau(Z)\mid \rho_0$ for all $Z\in  L^r$ hence $\rho_{C,r} \mid \rho_0$.
Thus $\rho_{C, r} = \rho_0$, as desired.
\end{proof}
Claim \ref{cl:min} together with Eq. \eqref{eq:rhoC} completes the proof.
\end{proof}

\noindent
\textbf{Acknowledgements.} 
The first author was supported by JSPS Grant-in-Aid for Scientists Research (C) 20K03513. 
 The second author was supported by JSPS Research Fellowship for Young Scientists Grant Number 19J12024, and is currently supported by a postdoctoral fellowship of the Alexander von Humboldt Foundation.
 The third author was supported by JSPS Grant-in-Aid for Scientists Research (B) 18H01115.
\bibliographystyle{alpha} 
\bibliography{references}

\end{document}